\documentclass[reqno]{amsart}
\usepackage{amssymb,latexsym,amsmath,amsthm,enumerate,amsbsy}
\usepackage[mathscr]{eucal}
\usepackage{framed,color,graphicx}
\usepackage{mathrsfs}
\usepackage[all]{xy}
\usepackage{tikz}
\usepackage{cite}
\usetikzlibrary{positioning,decorations.pathreplacing,patterns,decorations.pathmorphing}
\tikzset{%
element/.style={draw, shape=circle, fill=white, inner sep=1.4pt}
}

\DeclareSymbolFont{bbold}{U}{bbold}{m}{n}
\DeclareSymbolFontAlphabet{\mathbbold}{bbold}

\theoremstyle{plain}
\newtheorem{thm}{Theorem}[section]
\newtheorem{lem}[thm]{Lemma}
\newtheorem{lemma}[thm]{Lemma}

\newtheorem{corollary}[thm]{Corollary}
\newtheorem{pro}[thm]{Proposition}
\newtheorem{proposition}[thm]{Proposition}
\newtheorem{example}[thm]{Example}

\newtheorem{problem}[thm]{Problem}

\newtheorem{definition}[thm]{Definition}

\newtheorem{remark}[thm]{Remark}

\newcommand{\up}[1]{\textup{#1}}

\newcommand{\bc}{\mathbf{c}}

\newcommand{\bp}{\mathbf{p}}

\newcommand{\bu}{\mathbf{u}}
\newcommand{\bv}{\mathbf{v}}
\newcommand{\bw}{\mathbf{w}}

\begin{document}

\title[Flat semirings]
{The flat semirings with nilpotent multiplicative reducts}

\author{Zidong Gao}
\address{School of Mathematics, Northwest University, Xi'an, 710127, Shaanxi, P.R. China}
\email{zidonggao@yeah.net}

\author{Miaomiao Ren}
\address{School of Mathematics, Northwest University, Xi'an, 710127, Shaanxi, P.R. China}
\email{miaomiaoren@nwu.edu.cn}

\thanks{Miaomiao Ren, corresponding author, is supported by National Natural Science Foundation of China (12371024, 12571020).
}

\begin{abstract}
In this paper, we focus on the variety $\mathbf{NF}_3$ generated by all flat semirings
with $3$-nilpotent multiplicative reducts.
By introducing graph semirings, we characterize all subdirectly irreducible members of $\mathbf{NF}_3$.
We prove that the variety $\mathbf{NF}_3$ has uncountably many subvarieties
and show that every finitely generated subvariety of $\mathbf{NF}_3$ is a Cross variety.
Moreover,
we demonstrate that $\mathbf{NF}_3$ has a unique limit subvariety,
which is generated by all acyclic graph semirings.
\end{abstract}

\maketitle

\section{Introduction and preliminaries}\label{sec:intro}
An \emph{additively idempotent semiring} (ai-semiring for short) is an algebra $(S, +, \cdot)$ with two binary operations $+$ and $\cdot$
such that the additive reduct $(S, +)$ is a commutative idempotent semigroup,
the multiplicative reduct $(S, \cdot)$ is a semigroup and $S$ satisfies the distributive laws
\[
(x+y)z\approx xy+xz,\quad x(y+z)\approx xy+xz.
\]
This class of algebras includes (sometimes with extra unary operations or constants), the Kleene semiring of regular languages~\cite{con},
the max-plus algebra~\cite{aei}, the semiring of all binary relations on a set
\cite{dol09}, the matrix semiring over an ai-semiring~\cite{bg} and the distributive lattices~\cite{bs}.
These and other similar algebras have played important roles in several branches of mathematics,
such as algebraic geometry~\cite{cc}, tropical geometry~\cite{ms}, information science~\cite{gl}, and
theoretical computer science~\cite{go}.

A variety is \emph{finitely based} if it can be defined by finitely many identities;
otherwise, it is \emph{nonfinitely based}. An algebra $A$ is \emph{finitely based}
if the variety $\mathsf{V}(A)$ it generates is finitely based; otherwise, $A$ is \emph{nonfinitely based}.
In the last two decades,
the finite basis problem for ai-semirings have been intensively studied, for example,
see~\cite{dol07, gjrz, gpz05, jrz, pas05, rlzc, rlyc, rjzl, rzw, sr, shap23, vol21, wrz, wzr, yrzs, zrc}.
Dolinka~\cite{dol07} found the first example of a nonfinitely based finite ai-semiring and denoted it by $\Sigma_7$.
Pastijn et al.~\cite{gpz05, pas05} solved the finite basis problem for ai-semirings satisfying $x^2\approx x$.
Ren et al.~\cite{rzw} showed that every ai-semiring satisfying $x^3\approx x$ is finitely based.
Zhao et al.~\cite{zrc} showed that with the possible exception of $S_7$,
all ai-semirings of order three are finitely based.
The Cayley tables of $S_7$ can be found in Table~\ref{tb24111401}.
Jackson et al.~\cite{jrz} presented some general results about the finite basis problem for finite ai-semirings.
As applications, they showed that $S_7$ and $B_2^1$ whose multiplicative reduct is $6$-element Brandt monoid are both nonfinitely based.
We also note that Volkov~\cite{vol21} independently resolved the finite basis problem for $B^1_2$ and $\Sigma_7$ using a different method.
Recently, Gao et al.~\cite{gjrz} proved that the variety $\mathsf{V}(S_7)$ has uncountably many subvarieties.

\begin{table}[ht]
\caption{The Cayley tables of $S_7$} \label{tb24111401}
\begin{tabular}{c|ccc}
$+$      &$\infty$&$a$&$1$\\
\hline
$\infty$ &$\infty$&$\infty$&$\infty$\\
$a$      &$\infty$&$a$&$\infty$\\
$1$      &$\infty$&$\infty$&$1$\\
\end{tabular}\qquad
\begin{tabular}{c|ccc}
$\cdot$  &$\infty$&$a$&$1$\\
\hline
$\infty$ &$\infty$&$\infty$&$\infty$\\
$a$      &$\infty$&$\infty$&$a$\\
$1$      &$\infty$&$a$&$1$\\
\end{tabular}
\end{table}

By a \emph{flat semiring} $S$ we mean an ai-semiring such that its multiplicative reduct has a zero element $0$
and $a+b=0$ for all distinct elements $a$ and $b$ of $S$.
Jackson et al.~\cite[Lemma 2.2]{jrz}
observed that this definition of $+$ makes a semigroup $S$ with $0$ into a flat semiring if and only if
it is $0$-cancellative, that is, $ab=ac\neq0$ implies $b=c$ and $ba=ca\neq0$ implies $b=c$
for all $a, b, c\in S$. The algebra $S_7$ is an example of flat semirings.
The class of flat semirings has played an important role in the theory of ai-semiring varieties
(see~\cite{rjzl, rlyc, rlzc, yrzs}).
Let ${\bf F}$ denote the variety generated by all flat semirings.
The following result, which is due to Jackson et al.~\cite[Lemma 2.1]{jrz} (see also~\cite{jac:flat}),
solved the finite basis problem for ${\bf F}$
and characterized the subdirectly irreducible members of ${\bf F}$.
\begin{lem}\label{lem24121301}
The variety ${\bf F}$ is finitely based and each subdirectly irreducible member of ${\bf F}$ is a flat semiring.
\end{lem}

If $S$ is a cancellative semigroup, then $S^0$ is $0$-cancellative and becomes
a flat semiring, which is called the \emph{flat extension} of $S$. Let $G$ be a finite group.
Jackson~\cite[Theorem 7.3]{jac:flat} proved that the flat extension of $G$ is finitely based if and only if
all Sylow subgroups of $G$ are abelian.
Next, we introduce another important class of flat semirings.
Let $W$ be a nonempty subset of a free semigroup,
and let $S(W)$ denote the set of all nonempty subwords of words in
$W$ together with a new symbol $0$. If we define a binary operation $\cdot$ on $S(W)$ by the rule
\begin{equation*}
\bu\cdot \bv=
\begin{cases}
\bu\bv& \text{if }~\bu\bv\in S(W)\setminus \{0\}, \\
0& \text{otherwise,}
\end{cases}
\end{equation*}
then $(S(W), \cdot)$ forms a semigroup with zero element $0$.
It is easy to verify that $(S(W), \cdot)$ is $0$-cancellative, and so $S(W)$ becomes a flat semiring.
In particular, if $W$ consists of a single word $\bw$ we shall write $S(W)$ as $S(\bw)$.
We can also perform the same construction on a free commutative semigroup.
The resulting algebra is denoted by $S_c(W)$. Correspondingly, we have the notation $S_c(\bw)$.

A \emph{flat nil-semiring} $S$ is a flat semiring whose multiplicative reduct is a nil-semigroup,
that is, for any $a\in S$ there is an integer $n\geq 1$ such that $a^n=0$.
Let $k\geq 1$ be an integer.
Then a flat semiring $S$ is \emph{$k$-nilpotent} if its multiplicative reduct is $k$-nilpotent,
that is, $a_1\cdots a_k=0$ for all $a_1, \ldots, a_k\in S$.
A flat semiring is \emph{nilpotent} if it is $k$-nilpotent for some integer $k\geq 1$.
Note that $S(W)$ and $S_c(W)$ are both nilpotent if $W$ is finite.

The significance of flat semirings in the study of ai-semiring varieties stems from several key features.
First, they often serve as minimal counterexamples, providing tractable yet nontrivial test cases
for equational properties. Second, their congruence lattices are completely determined by
multiplicative ideals, making them particularly amenable to structural analysis. Third, despite
their apparent simplicity, flat semirings exhibit rich behaviour with respect to the finite basis
problem: they include both finitely based and nonfinitely based examples, and their subvariety
structure can be remarkably complex--a fact vividly illustrated by the aforementioned results
on $\mathsf{V}(S_7)$. These features make the class of flat semirings a natural proving ground
for developing techniques that may apply more broadly in the study of ai-semiring varieties.

A significant step in this direction was taken by Jackson et al.~\cite[Corollary 4.11]{jrz},
who proved that every finite flat semiring whose variety contains $S_7$ is nonfinitely based.
In light of this result, they posed the following fundamental problem~\cite[Problem 7.1(4)]{jrz}:
\begin{problem}\label{problem29}
Which finite flat semirings are finitely based?
\end{problem}
\noindent
Up to now, this problem has not been fully resolved.
In particular, Gao et al.~\cite{gjrz} showed that every flat semiring
whose variety contains $S_c(abc)$ and is included in
the variety $\mathsf{V}(S_7)$ has no finite basis for its equational theory.
Wu et al.~\cite{wzr} completely solved the finite basis problem for finite nilpotent flat semirings of the form $S_c(W)$,
and proved that some finite nilpotent flat semirings of the form $S(W)$ are nonfinitely based.
Ren et al.~\cite{rlzc} and Shaprynski\v{\i}~\cite{shap23} solved the problem for $4$-element flat semirings.
The present paper is another contribution to Problem~\ref{problem29}.
We shall show that every finite $3$-nilpotent flat semiring is finitely based.

For any integer $k \geq 1$,
let $\mathbf{NF}_k$ denote the variety generated by all $k$-nilpotent flat semirings.
\begin{proposition}\label{fkjidi}
Let $k\geq 1$ be an integer. Then
\begin{enumerate}[$(\rm i)$]
\item$\mathbf{NF}_k$ is the subvariety of $\mathbf{F}$
determined by the $k$-nilpotent identity
\begin{equation}\label{id25052201}
x_1\ldots x_k \approx y_1\cdots y_k.
\end{equation}

\item $\mathbf{NF}_k$ is finitely based.

\item $\mathbf{NF}_k$ is a proper subvariety of $\mathbf{NF}_{k+1}$.
\end{enumerate}
\end{proposition}

\begin{proof}
(i) Let ${\mathcal V}_k$ denote the subvariety of $\mathbf{F}$ determined by the identity \eqref{id25052201}.
It is easy to see that $\mathbf{NF}_k$ is a subvariety of ${\mathcal V}_k$.
Conversely, let $S$ be an arbitrary subdirectly irreducible member of ${\mathcal V}_k$.
By Lemma~\ref{lem24121301} we deduce that $S$ is a $k$-nilpotent flat semiring
and so $S$ is a member of $\mathbf{NF}_k$.
Thus ${\mathcal V}_k$ is a subvariety of $\mathbf{NF}_k$
and so $\mathbf{NF}_k=\mathcal{V}_k$ as required.

(ii) This follows from (i) and Lemma~\ref{lem24121301} immediately.

(iii) It is easy to see that every $k$-nilpotent flat semiring is $(k+1)$-nilpotent.
So $\mathbf{NF}_k$ is a subvariety of $\mathbf{NF}_{k+1}$.
Notice that the flat semiring $S(a^k)$ is a member of $\mathbf{NF}_{k+1}$, but does not lie in $\mathbf{NF}_k$.
We therefore have that $\mathbf{NF}_k$ is proper in $\mathbf{NF}_{k+1}$.
\end{proof}

It is easy to see that $\mathbf{NF}_1$ is the trivial variety.
Corollary~\ref{coro25052801} will show that $\mathbf{NF_2}$
is generated by the $2$-element flat semiring $S(a)$ and is a minimal nontrivial ai-semiring variety.
Jackson et al.~\cite{gjrz} introduced block hypergraph semirings and the associated ai-semiring term
$t_{\mathbb{H}}$ for a hypergraph $\mathbb{H}$.
Using these algebras and terms, they fully solved the finite basis problem for all subvarieties of $\mathsf{V}(S_7)$,
completely characterized all finite subdirectly irreducible members of $\mathsf{V}(S_7)$,
and proved that the intersection of $\mathsf{V}(S_7)$ and $\mathbf{NF}_k$
has uncountably many subvarieties for all $k\geq 4$.
Motivated by this work, we introduce graph semirings in Section~2
and then apply them to fully describe all subdirectly irreducible members of $\mathbf{NF}_3$
and show that $\mathbf{NF}_3$ also has uncountably many subvarieties.
In Section~3 we prove that
every finitely generated subvariety of $\mathbf{NF}_3$ is a Cross variety.

A variety is \emph{hereditarily finitely based} if all of its subvarieties are finitely based.
A \emph{limit variety} means a minimal nonfinitely based subvariety.
By Zorn's lemma, every nonfinitely based variety contains a limit variety.
So a variety is hereditarily finitely based if and only if it contains no limit subvarieties.
Therefore, classifying hereditarily finitely based varieties in a certain sense reduces to
classifying limit varieties.
As Lee and Volkov~\cite{lv} stated,
the latter task is generally quite challenging, and even constructing explicit examples of limit varieties proves to be highly nontrivial.
Ren et al.~\cite{rjzl} used the flat extensions of groups to
provide an infinite family of limit ai-semiring varieties.
They also showed that the variety $\mathsf{V}(S_c(abc))$ is a limit subvariety of
the variety $\mathsf{V}(S_7)$.
Gao et al.~\cite{gjrz} proved that $\mathsf{V}(S_c(abc))$ is the unique limit subvariety of $\mathsf{V}(S_7)$.
All of these limit varieties have finitely many subvarieties.
In Section~4 we present a new limit ai-semiring variety,
which is the unique limit subvariety of $\mathbf{NF}_3$ and has countably infinitely many subvarieties.

Recall that a nontrivial algebra is \emph{subdirectly irreducible} if it has a least non-diagonal congruence.
Let $S$ be a flat semiring. One can easily verify that there is a one-to-one
order-preserving correspondence between semiring congruences on $S$ and multiplicative ideals of $S$.
So $S$ is subdirectly irreducible if and only if it has a least nonzero multiplicative ideal.
A nonzero element $\omega$ of $S$ is an \emph{annihilator} if $\omega s=s\omega=0$ for all $s\in S$.
It follows directly from the definition that the annihilators of $S(W)$ and $S_c(W)$ are precisely maximal words in $W$.
The following result provides another characterization of subdirectly irreducible flat nil-semirings
in terms of annihilators.

\begin{pro}\label{sdnil}
Let $S$ be a flat nil-semiring.
Then $S$ is subdirectly irreducible if and only if there exists an annihilator $\omega$ of $S$ such that $\omega\in S^1aS^1$ for all nonzero elements $a$ of $S$.
In this case, $\omega$ is the unique annihilator of $S$ and $\{0,\omega\}$ is the least nonzero multiplicative ideal of $S$.
\end{pro}
\begin{proof}
Suppose that there exists an annihilator $\omega$ of $S$ such that $\omega\in S^1aS^1$ for all nonzero elements $a$ of $S$.
Then $\{0, \omega\}$ is clearly a nonzero multiplicative ideal of $S$.
Let $J$ be an arbitrary nonzero multiplicative ideal of $S$. Take a nonzero element $a$ of $J$.
Then $\omega$ is in $S^1aS^1$. Since $S^1aS^1\subseteq J$, it follows that $\omega$ is in $J$
and so $\{0, \omega\}\subseteq J$. This shows that $\{0,\omega\}$ is the least nonzero multiplicative ideal of $S$.
So $S$ is subdirectly irreducible. Furthermore,
if $\omega'$ is an arbitrary annihilator of $S$,
then $\{0,\omega'\}$ is a multiplicative ideal of $S$ and so $\{0,\omega\}\subseteq \{0,\omega'\}$.
Thus $\omega=\omega'$ and so $\omega$ is the unique annihilator of $S$.

Conversely, suppose that $S$ is subdirectly irreducible.
Then $S$ has a least nonzero multiplicative ideal $I$.
Let $a$ and $b$ be nonzero elements of $I$.
Then $S^{1}aS^{1}=S^{1}bS^{1}=I$. This implies that $a=xby$ and $b=uav$ for some $x,y,u,v\in S^{1}$.
So
\[
a=(xu)a(vy)=(xu)^2a(vy)^2=\cdots.
\]
Since $(S, \cdot)$ is a nil-semigroup, it is not hard to prove that $x=u=v=y=1$.
Hence $a=b$ and so $|I|=2$. We may write $I=\{0,\omega\}$.
Next, we shall show that $\omega$ is an annihilator of $S$.
Indeed, let $s$ be an element of $S$. Then $s\omega\in I$ and so $s\omega$ ie equal to $0$ or $\omega$.
If $s\omega=\omega$ then $s^n\omega=\omega$ for all $n\geq 1$.
As $(S, \cdot)$ is a nil-semigroup, we have that $s^n=0$ for some $n\geq 1$.
Thus $\omega=0$, a contradiction.
So $s\omega = 0$. Similarly, $\omega s=0$.
We have shown that $\omega$ is an annihilator of $S$.
Since $\{0,\omega\}$ is the least nonzero multiplicative ideal of $S$,
it follows immediately that $\omega\in S^1aS^1$ for all nonzero elements $a$ of $S$.
\end{proof}

\begin{proposition}\label{sinilpotent}
A nilpotent flat semiring is subdirectly irreducible if and only if it has a unique annihilator.
\end{proposition}

\begin{proof}
Suppose that $S$ is a nilpotent flat semiring that has a unique annihilator $\omega$.
For any nonzero element $a$ of $S$, let $m$ denote the number
\[
\max\{k\geq 1 \mid  (\exists a_1, a_2, \ldots, a_k\in S)~a_1\cdots a_k\neq 0, a\in \{a_1, \ldots, a_k\}\}.
\]
Then there exist $a_1, \ldots, a_{i-1}, a_{i+1}, \ldots, a_m\in S$ such that
$a_1\cdots a_{i-1}aa_{i+1}\cdots a_m$, which is denoted by $u$, is nonzero.
The maximality of $m$ ensures that that $u$ is an annihilator of $S$ and so $u=\omega$, since $\omega$ is the unique annihilator of $S$.
Thus $\omega\in S^1aS^1$ and so is subdirectly irreducible by Proposition~\ref{sdnil}.
The converse part is a consequence of Proposition~\ref{sdnil}.
\end{proof}

The following example shows that a flat nil-semiring that has a unique annihilator need not be subdirectly irreducible.
\begin{example}
Let $W = \{x_1x_2\cdots x_n\mid n\geq 1\}\cup\{x_0\}$.
Then $S(W)$ is a flat nil-semiring $($satisfying $x^2\approx 0$$)$ with unique annihilator $x_0$, but it is not subdirectly irreducible:
the set $S(W)\backslash \{x_0\}$ is a multiplicative ideal avoiding $x_0$, so $\{0,x_0\}$ is not the least nonzero ideal, contradicting the criterion of Proposition~$\ref{sdnil}$.
\end{example}

\begin{lemma}\label{SkSk+1}
Let $S$ be a nilpotent flat semiring that has a unique annihilator $\omega$.
Then $S^{k}=\{0,\omega\}$ for some $k\geq 1$.
\end{lemma}
\begin{proof}
By assumption, there exists an integer $k\geq 1$ such that $S^{k}\neq \{0\}$ and $S^{k+1}= \{0\}$.
The maximality of $k$ ensures that every nonzero element of $S^{k}$ is an annihilator of $S$.
Since $\omega$ is the unique annihilator of $S$, it follow immediately that $S^{k}=\{0, \omega\}$.
\end{proof}

\begin{corollary}\label{coro25052801}
The variety $\mathbf{NF_2}$ is generated by $S(a)$ and is a minimal nontrivial ai-semiring variety.
\end{corollary}
\begin{proof}
Let $S$ be a nontrivial subdirectly irreducible member of $\mathbf{NF_2}$.
It follows from Lemma~\ref{lem24121301} and Proposition~\ref{sinilpotent} that $S$ has a unique annihilator $\omega$.
By Lemma~\ref{SkSk+1} we deduce that $S=\{0, \omega\}$.
Now it is a routine matter to verify that $S$ is isomorphic to $S(a)$.
This shows that up to isomorphism,
$S(a)$ is the only nontrivial subdirectly irreducible member of $\mathbf{NF_2}$.
So $\mathbf{NF_2}$ is generated by $S(a)$.
From \cite{sr} we know that $\mathsf{V}(S(a))$ is a minimal nontrivial ai-semiring variety.
This proves the required result.
\end{proof}

\begin{definition}
A flat semiring $S$ is the $0$-direct union of a family $(S_i)_{i\in I}$ of flat semirings,
denoted by $S=\bigcup_{i\in I}^{\bullet} S_i$,
if
\[
\textstyle S=\bigcup_{i\in I} S_i'
\]
and
\[
S_i'\cong S_i, S_j'\cap S_k'=S_j'\cdot S_k'=\{0\}
\]
for all $i, j, k\in I$ with $j\neq k$.
\end{definition}

Let $\mathcal{K}$ be a class of ai-semirings.
Then $\mathsf{V}(\mathcal{K})$ denotes the variety generated by $\mathcal{K}$, that is,
the smallest variety that contains $\mathcal{K}$.
Let $(S_i)_{i\in I}$ be a class of flat semirings. Then
it is easy to verify that $\bigcup_{i\in I}^{\bullet} S_i$ is isomorphic to a subdirect product of $(S_i)_{i\in I}$.
So $\mathsf{V}(\{S_i\mid i\in I\})=\mathsf{V}(\bigcup_{i\in I}^{\bullet} S_i)$.
Together with Lemma~\ref{lem24121301}, we have the following proposition.
\begin{proposition}
Every subvariety of $\mathbf F$ can be generated by a single flat semiring.
In particular, every finitely generated subvariety of $\mathbf F$ can be generated by a finite flat semiring.
\end{proposition}

\begin{definition}
A flat semiring $S$ is the $\{0, \omega\}$-direct union of a family $(S_i)_{i\in I}$ of subdirectly irreducible flat nil-semirings,
denoted by $\bigcup_{i\in I}^{\omega} S_i$,
if
\[
\textstyle S=\bigcup_{i\in I} S_i'
\]
and
\[
S_i'\cong S_i, S_j'\cap S_k'=\{0,\omega\}, S_j'\cdot S_k'=\{0\}
\]
for all $i, j, k\in I$ with $j \neq k$.
\end{definition}
A straightforward verification using the definition above yields the following proposition.
\begin{proposition}
Suppose that $\{S_i\}_{i\in I}$ is a family of subdirectly irreducible flat nil-semirings with pairwise intersection $\{0,\omega\}$.
If $J\subseteq I$ and $S_i$ is a subalgebra of $T_i$ containing $\omega$ for each $i\in J$, then $\bigcup_{i\in J}^{\omega} S_i$ is a subalgebra of $\bigcup_{i\in I}^{\omega} T_i$.
\end{proposition}

\begin{remark}
The crucial condition that each $S_i$ must share the same distinguished element $\omega$ with $T_i$. This point deserves special emphasis; for instance, $S(a^2)\circ S(a^2)$ is not a subalgebra of $S(a^2b)\circ S(a^2)$, because the latter satisfies the identity $x^2+y^2\approx x^2+y^2+xy$ while the former does not.
\end{remark}

In the sequel we sometimes denote by $S_1\circ S_2$ the $\{0, \omega\}$-direct union of
the subdirectly irreducible flat nil-semirings $S_1$ and $S_2$.
\begin{proposition}
The $\{0, \omega\}$-direct union of a family of subdirectly irreducible flat nil-semirings is also subdirectly irreducible.
\end{proposition}
\begin{proof}
Suppose that $S$ is the $\{0, \omega\}$-direct union of a family $(S_i)_{i\in I}$ of subdirectly irreducible flat nil-semirings.
Then $S$ is also a flat nil-semiring containing a unique annihilator $\omega$.
For each $a\in S\backslash \{0\}$, we have that  $a\in S_i$ for some $i\in I$.
It follows from Proposition~\ref{sdnil} that $\omega\in S_i^1aS_i^1\subseteq S^1aS^1$.
By Proposition~\ref{sdnil} again $S$ is subdirectly irreducible as required.
\end{proof}

The following example tells us that a subdirectly irreducible flat nil-semiring is not necessarily nilpotent.
\begin{example}
The $\{0, \omega\}$-direct union of $\{S_c(a_{1}\cdots a_{k})\mid k\geq 1\}$ is a subdirectly irreducible flat nil-semiring
but is not nilpotent.
\end{example}

Let $X^+$ denote the free semigroup over a countably infinite set $X$ of variables.
By distributivity, all ai-semiring terms over $X$ are finite sums of words in $X^+$.
An \emph{ai-semiring identity} over $X$ is a formal expression of the form
\[
\bu\approx \bv,
\]
where $\bu$ and $\bv$ are ai-semiring terms over $X$.
Let $S$ be an ai-semiring, and let $\bu\approx \bv$ be an ai-semiring identity over $\{x_1, x_2, \ldots, x_n\}$.
Then $S$ \emph{satisfies} $\bu\approx \bv$ if
$\bu(a_1, a_2, \ldots, a_n)=\bv(a_1, a_2, \ldots, a_n)$ for all $a_1, a_2, \ldots, a_n\in S$, where $\bu(a_1, a_2, \ldots, a_n)$ denotes the result of evaluating $\bu$ in $S$ under the assignment $x_i\mapsto a_i$, and similarly for $\bv(a_1, a_2, \ldots, a_n)$.

For other notations and terminology used in this paper, the reader is referred to
Jackson et al.~\cite{jrz} and Ren et al.~\cite{rjzl} for background on semirings,
and to Burris and Sankappanavar~\cite{bs} for information concerning the theory of varieties.
We shall assume that the reader is familiar with the basic results in these areas.

\section{The number of subvarieties of $\mathbf{NF}_3$}
In this section we prove that the variety $\mathbf{NF}_3$ has uncountably many subvarieties.
Throughout this paper, all graphs are directed, contain no isolated vertices,
and satisfy the condition that every vertex has both out-degree and in-degree at most $1$.
One can easily deduce that every connected component of such a graph is either a path or a cycle.
Let ${\mathbb G}=\langle V(\mathbb G), E(\mathbb G)\rangle$ be a graph, where $V(\mathbb G)$ is the vertex set
and $E(\mathbb G)$ is the edge set.
Let $S_{\mathbb G}$ denote the disjoint union of $V(\mathbb G)$ and $\{0, \omega\}$.
Define a multiplicative operation $\cdot$ on $S_{\mathbb G}$ by the rule
\[
x\cdot y=\begin{cases}
    \omega,& (x,y)\in E(\mathbb G),\\
    0,&\text{otherwise}.
  \end{cases}
\]
It is straightforward to verify that $(S_{\mathbb G},\cdot)$ forms a $0$-cancellative semigroup and is $3$-nilpotent,
and so $S_{\mathbb G}$ becomes a flat semiring in $\mathbf{NF}_3$.
We call $S_{\mathbb G}$ the \emph{graph semiring} induced by ${\mathbb G}$.
For instance, let $\mathbb G$ be the path $a \rightarrow b \rightarrow c$. Then $S_{\mathbb G} = \{0, \omega, a, b, c\}$ with $ab = \omega = bc$, and all other products among vertices equal to $0$. This instantly clarifies the connection between graph adjacency and multiplication.
In particular, if the graph ${\mathbb G}$ is empty, then $S_{\mathbb G}$ is isomorphic to $S(a)$.
The following result demonstrates the importance of graph semirings.
\begin{thm}\label{siSg}
A nontrivial $3$-nilpotent flat semiring is subdirectly irreducible if and only if it is isomorphic to some graph semiring.
\end{thm}
\begin{proof}
Let $S_{\mathbb G}$ be a graph semiring.
Then $S_{\mathbb G}$ is a $3$-nilpotent flat semiring that has a unique annihilator $\omega$.
By Proposition~\ref{sinilpotent} we deduce that $S_{\mathbb G}$ is subdirectly irreducilbe.
Conversely, suppose that $S$ is a subdirectly irreducible $3$-nilpotent flat semiring.
By Proposition~\ref{sinilpotent} and Lemma~\ref{SkSk+1},
$S$ contains a unique annihilator $\omega$, and $S^2$ is contained in $\{0,\omega\}$.
Define a graph $\mathbb G$ with the vertex set $V(\mathbb G)=S\setminus\{0,\omega\}$
and the edge set $E(\mathbb G)=\{(x,y)\mid xy=\omega\}$.
Now it is easy to verify that $S$ is isomorphic to the graph semiring $S_{\mathbb G}$.
This proves the required result.
\end{proof}

\begin{definition}
Let $m\geq 2$ and $n\geq 1$ be integers, and let
\[
\bp_m=x_1x_2+x_2x_3+\cdots+x_{m-1}x_m
\]
and
\[
\bc_n=x_1x_2+x_2x_3+\cdots+x_{n-1}x_n+x_nx_1
\]
be ai-semiring terms.
Then the path graph semiring $S_{\bp_m}$ denotes
the graph semiring induced by the path graph
\[
\langle\{a_i \mid 1\leq i \leq m\}, \{(a_i, a_{i+1})\mid 1\leq i <m\}\rangle,
\]
and the cycle graph semiring $S_{\bc_n}$ denotes the graph semiring induced by
the cycle graph
\[
\langle\{a_i \mid 1\leq i \leq n\}, \{(a_i, a_{i+1})\mid 1\leq i <n\}\cup \{(a_n, a_1)\}\rangle.
\]
\end{definition}

\begin{remark}
$S_{\bp_2}\cong S(ab)$, $S_{\bc_1}\cong S(a^2)$, and $S_{\bc_2}\cong S_c(ab)$.
\end{remark}
Notice that as we required graphs to have all vertex indegrees
and outdegrees at most $1$ it follows that
 every connected component of a graph is either a path or a cycle.
By the construction of graph semirings we immediately have the following result,
which will be repeatedly used in the sequel.
\begin{pro}\label{pro25061320}
The $\{0, \omega\}$-direct union of graph semirings is a graph semiring. Conversely,
if ${\mathbb G}$ is a graph and $\{{\mathbb G}_i\}_{i\in I}$ denotes the set of all connected components of ${\mathbb G}$,
then $S_{\mathbb G}$ is the $\{0, \omega\}$-direct union of $\{S_{\mathbb G_i}\}_ {i\in I}$,
where each $S_{\mathbb G_i}$ is either a path graph semiring or a cycle graph semiring.
\end{pro}

\begin{lemma}\label{Scnm}
Let $m$ and $n$ be positive integers. Then
$S_{\bc_n}$ satisfies $\bc_m\approx x^3$ if and only if $n$ does not divide $m$.
\end{lemma}
\begin{proof}
Suppose that $n$ divides $m$. Then $m=nk$ for some integer $k$.
Consider the substitution $\varphi: \{x, x_1, \ldots, x_m\} \to S_{\bc_n}$ such that
$\varphi(x)=0$ and $\varphi(x_{rn+s})=a_s$ for all $0\leq r<k$ and $1\leq s \leq n$.
It is easy to see that $\varphi(\bc_m)=\omega$ and $\varphi(x^3)=0$.
So $S_{\bc_m}$ does not satisfy $\bc_m\approx x^3$.

Conversely, assume that $S_{\bc_n}$ does not satisfy $\bc_m\approx x^3$.
Then there is a substitution $\psi: \{x, x_1, \ldots, x_m\} \to S_{\bc_n}$ such that
$\psi(\bc_m)\neq 0$, since $\psi(x^3)=0$. So we have
\[
\psi(x_1)\psi(x_2)=\cdots =\psi(x_{m-1})\psi(x_m)=\psi(x_m)\psi(x_1)=\omega.
\]
For any $1\leq i \leq m$, let $a_{j_i}$ denote $\psi(x_i)$, $1\leq j_i \leq n$. Then
\begin{equation}\label{eq25053001}
j_{k+1} \equiv j_k+1 \pmod n
\end{equation}
for all $1\leq k <m$ and
\begin{equation}\label{eq25053002}
j_1\equiv j_{m}+1 \pmod n.
\end{equation}
By \eqref{eq25053001} one can deduce that
\begin{equation}\label{eq25053003}
j_{m} \equiv j_1+m-1 \pmod n.
\end{equation}
Combining \eqref{eq25053002} and \eqref{eq25053003} we obtain that $n$ divides $m$ as required.
\end{proof}

In what follows we shall use $P$ to denote the set of all prime numbers.
\begin{thm}\label{cor:continuum}
The lattice $\mathcal{L}(\mathbf{NF}_3)$ of subvarieties of $\mathbf{NF}_3$ has cardinality of the continuum.
Moreover, it contains both a chain and an antichain of cardinality $2^{\aleph_0}$.
\end{thm}

\begin{proof}
Let $\mathcal{P}(P)$ denote the power set of $P$.
One can show that the mapping $\varphi: \mathcal{P}(P) \to \mathcal{L}(\mathbf{NF}_3)$ defined by
\[
\varphi(Q)=\mathsf{V}(\{S_{\bc_q}\mid q\in Q\}) \quad (Q\in \mathcal{P}(P)),
\]
is a lattice embedding mapping.
Indeed, let $Q$ a proper subset of $P$, and let $p$ be an element of $P\backslash Q$.
By Lemma~\ref{Scnm} we have that $\bc_p\approx x^3$ is satisfied by every semiring in $\{S_{\bc_q}\mid q\in Q\}$,
but does not hold in $S_{\bc_p}$.
By this observation, we can easily deduce that $\varphi$ is a lattice embedding mapping.
Since $\mathcal{P}(P)$ has the cardinality of the continuum, it follows immediately that $\mathcal{L}(\mathbf{NF}_3)$ inherits this cardinality.
From \cite[Chapter I, \S10]{ku2011} we know that $\mathcal{P}(P)$ has a chain and an antichain of size $2^{\aleph_0}$.
Consequently, $\mathcal{L}(\mathbf{NF}_3)$ inherits this property.
\end{proof}
As a consequence, we obtain the following.
\begin{corollary}
Let $k\geq 4$ be an integer. Then
$\mathbf{NF}_k$ has uncountably many subvarieties.
\end{corollary}

Theorem \ref{cor:continuum} established the uncountable cardinality of $\mathcal{L}(\mathbf{NF}_3)$ via semantic methods.
In contrast, Proposition \ref{pro25053101} below provides a syntactic proof using equationally definable subvarieties.

\begin{proposition}\label{pro25053101}
Let $I$ be an arbitrary subset of $P$, and
let $\mathcal{V}_I$ denote the subvariety of $\mathbf{NF}_3$ defined by the identities
\[
\bc_q\approx x^3, \quad q\in I.
\]
Then $\mathcal{V}_{I_1}\neq \mathcal{V}_{I_2}$
whenever $I_1$ and $I_2$ are distinct subsets of $P$.
\end{proposition}
\begin{proof}
Let $I_1$ and $I_2$ are distinct subsets of $P$.
Without loss of generality, we may assume that $I_1$ is not contained in $I_2$.
Then there exists an element $q$ of $I_1$ that is not in $I_2$.
By virtue of Lemma~\ref{Scnm}, we deduce that $S_{\bc_q}$ is a member of $\mathcal{V}_{I_2}$,
but does not lie in $\mathcal{V}_{I_1}$.
Thus $\mathcal{V}_{I_1} \neq \mathcal{V}_{I_2}$.
\end{proof}

\begin{corollary}
Let $I$ be a subset of $P$.
Then $\mathcal{V}_I$ is nonfinitely based if and only if $I$ is infinite.
In this case, $\mathcal{V}_I$ has an irredundant equational basis.
\end{corollary}
\begin{proof}
If $I$ is finite, then it follows from Proposition \ref{fkjidi}(ii) that $\mathcal{V}_I$ is finitely based.
Now let $I$ be an infinite subset of $P$. Suppose by way of contradiction that $\mathcal{V}_I$ is finitely based.
By the Compactness Theorem for Equational Logic~\cite[Exercise 10 in \S14]{bs},
it follows that $\mathcal{V}_I=\mathcal{V}_J$ for some finite subset $J$ of $I$,
which contradicts Proposition~\ref{pro25053101}.
So $\mathcal{V}_I$ is nonfinitely based.
Moreover, Proposition~\ref{pro25053101} also implies that $\{\bc_q\approx x^3\mid q\in I\}$
is an irredundant equational basis of $\mathcal{V}_I$ within $\mathbf{NF}_3$.
\end{proof}

\section{The Cross subvarieties of $\mathbf{NF}_3$}
In this section we show that every finitely generated subvariety of $\mathbf{NF}_3$ is a Cross variety.
Recall that an algebra $D$ is a \emph{divisor} of an algebra $A$ is $D$ is a homomorphic image of
a subalgebra $B$ of $A$. The divisor $D$ is \emph{proper} if either $B$  is a proper
subalgebra of $A$ or the homomorphism from $B$ onto $D$ is not an isomorphism.
A finite algebra is \emph{critical} if it does not lie in the variety generated by all of its proper divisors.
It is easily verified that every critical algebra is subdirectly irreducible.
A \emph{Cross} variety is a variety that is locally finite, finitely based, and has finitely many critical algebras.
By \cite[Theorem 2.1]{mv}, a variety is a Cross variety if and only if it is finitely generated, finitely based,
and has finitely many subvarieties.
From~\cite[Proposition 1.4.35]{sapir14} we know that
every subvariety of a Cross variety is a Cross variety.
Let $\mathcal{V}_1$ and $\mathcal{V}_2$ be varieties. We write $\mathcal{V}_1 < \mathcal{V}_2$
if $\mathcal{V}_1$ is a proper subvariety of $\mathcal{V}_2$.

The following lemma establishes key relationships among varieties generated by graph semirings.
These relationships will be used repeatedly throughout this section to analyze the structure
of subvarieties of $\mathbf{NF}_3$.
%
%
%
%
%

\begin{lemma}\label{scnsln1}
Let $k\geq 1$ and $m,n\geq 2$ be integers. Then
\begin{enumerate}[$(\rm i)$]
\item \label{eq250602011}
$\mathsf{V}(S_{\bp_{n}})<\mathsf{V}(S_{\bp_{n}}\circ S_{\bp_n})
=\mathsf{V}\left(\bigcup_{1\leq i\leq m}^{\omega}S_{\bp_{n}}\right)$,

\item \label{eq250602012}
$\mathsf{V}(S_{\bp_{n}}\circ S_{\bp_n})<\mathsf{V}(S_{\bp_{n+1}})
=\mathsf{V}\left(S_{\bp_{n+1}}\circ \bigcup_{1\leq i\leq m}^{\omega}S_{\bp_{n}}\right)$,

\item\label{eq25060202}
$\mathsf{V}(S_{\bc_1})<\mathsf{V}(S_{\bc_1}\circ S_{\bc_1})
=\mathsf{V}\left(\bigcup_{1\leq i\leq m}^{\omega}S_{\bc_1}\right)$,

\item \label{eq25060203}
$\mathsf{V}\left(\bigcup_{1\leq i\leq m}^{\omega}S_{\bc_n}\right)
=\mathsf{V}(S_{\bc_n})$,

\item \label{eq26216}
$\bigcup_{1\leq i\leq m}^{\omega}S_{\bp_n} \in \mathsf{V}(S_{\bc_k}\circ S_{\bp_n}).$
\end{enumerate}
\end{lemma}

\begin{proof}
(i) Since $S_{\bp_{n}}$ is a subalgebra of $S_{\bp_{n}}\circ S_{\bp_{n}}$,
it follows that $\mathsf{V}(S_{\bp_{n}})$ is a subvariety of $\mathsf{V}(S_{\bp_{n}}\circ S_{\bp_n})$.
Let
\begin{equation}\label{id25060820}
\bp_n(x_1, \ldots, x_n)+\bp_n(y_1, \ldots, y_n)\approx \bp_n(x_1, \ldots, x_{n-1},y_n)+\bp_n(y_1, \ldots, y_{n-1},x_n)
\end{equation}
be an ai-semiring identity.
It is easy to check that the identity~\eqref{id25060820} is satisfied by $S_{\bp_{n}}$,
but does not hold in $S_{\bp_{n}}\circ S_{\bp_n}$.
So $\mathsf{V}(S_{\bp_{n}})$ is a proper subvariety of $\mathsf{V}(S_{\bp_{n}}\circ S_{\bp_n})$.

It is obvious that
$\mathsf{V}(S_{\bp_{n}}\circ S_{\bp_n})$ is a subvariety of $\mathsf{V}\left(\bigcup_{1\leq i\leq m}^{\omega}S_{\bp_{n}}\right)$,
since $S_{\bp_{n}}\circ S_{\bp_n}$ is a subalgebra of $\bigcup_{1\leq i\leq m}^{\omega}S_{\bp_{n}}$.
To establish the reverse inclusion, we need to show that $\bigcup_{1\leq i\leq m}^{\omega}S_{\bp_{n}}$ lies in $\mathsf{V}(S_{\bp_{n}}\circ S_{\bp_n})$.
Let $\{0, \omega, a_1, \ldots, a_n\}$ and $\{0, \omega, b_1, \ldots, b_n\}$ be copies of $S_{\bp_n}$.
Then their $\{0, \omega\}$-direct union $\{0, \omega, a_1, b_1, \ldots, a_n, b_n\}$,
denoted by $T$, is isomorphic to $S_{\bp_{n}}\circ S_{\bp_n}$.

The idea is to simulate $m$ disjoint copies of $S_{\bp_n}$ inside the direct power $T^m$.
We use the generators $(a_i, b_i, \ldots, b_i)$ and their cyclic shifts to ensure that each copy of $S_{\bp_n}$ occupies a different coordinate position.
Specifically, for each $j$, the $j$-th cyclic shift $\sigma^j(a_i,b_i,\ldots,b_i)$ has its non-zero entries concentrated in the $j$-th coordinate (where it behaves like $a_i$) and the other coordinates (where it behaves like $b_i$) are structured so that products between different shifts vanish.
The subalgebra generated by these elements then contains an isomorphic copy of each $S_{\bp_n}$ in each coordinate, but also contains additional elements.
Taking the quotient by the ideal $J$ of elements with a zero coordinate collapses these extra elements and yields precisely the $\{0,\omega\}$-direct union of $m$ copies of $S_{\bp_n}$.

To realize this construction formally, let $A$ denote the subalgebra of the direct product $T^m$ of $m$ copies of $T$ generated by the elements
\[
\sigma^j(a_i,b_i, \ldots, b_i),~1\leq i\leq n,~ 1\leq j\leq m,
\]
where $\sigma: T^m \to T^m$ is the right cyclic shift mapping defined by
\[
\sigma(x_1, x_2, \ldots, x_m)=(x_m, x_1, \ldots, x_{m-1}).
\]
If we denote by $J$ the set of all elements of $A$ with a $0$-coordinate,
then $J$ is both a multiplicative ideal and an order-theoretic filter of $A$.
Now it is a routine matter to verify that the quotient algebra $A/J$ is isomorphic to $\bigcup_{1\leq i\leq m}^{\omega} S_{\bp_n}$.
Thus
$\mathsf{V}(S_{\bp_{n}}\circ S_{\bp_n})=\mathsf{V}\left(\bigcup_{1\leq i\leq m}^{\omega}S_{\bp_{n}}\right)$.

(ii) Suppose that $S_{\bp_{n+1}}=\{0, \omega, a_1, \ldots, a_{n+1}\}$.
We now show that $S_{\bp_{n+1}}\circ \bigcup_{1\leq i\leq m}^{\omega}S_{\bp_n}$ belongs to $\mathsf{V}(S_{\bp_{n+1}})$.
The construction uses diagonal elements to simulate $S_{\bp_{n+1}}$ and cyclically shifted tuples to simulate the $m$ copies of $S_{\bp_n}$.

Formally, let $A$ be the subalgebra of the direct product $S_{\bp_{n+1}}^{m}$ of $m$ copies of $S_{\bp_{n+1}}$
generated by
\[
(a_i,a_i,\ldots,a_i),~1\leq i\leq n+1,
\]
\[
\sigma^j(a_i,a_{i+1}, \ldots, a_{i+1}),~1\leq i\leq n,~ 1\leq j\leq m,
\]
where $\sigma$ is the right cyclic shift mapping defined by $S_{\bp_{n+1}}^{m}$ onto itself.
Let $J$ denote the set of all elements of $A$ with a $0$-coordinate.
Then $J$ is both a multiplicative ideal and an order-theoretic filter of $A$.
It is easy to verify that $A/J$ is isomorphic to
$S_{\bp_{n+1}}\circ \bigcup_{1\leq i\leq m}^{\omega}S_{\bp_n}$.
Thus
\[
\textstyle\mathsf{V}\left(S_{\bp_{n+1}}\circ \bigcup_{1\leq i\leq m}^{\omega}S_{\bp_n}\right)=\mathsf{V}(S_{\bp_{n+1}}).
\]
In particular, $\mathsf{V}(S_{\bp_n}\circ S_{\bp_n})$ is a subvariety of $\mathsf{V}(S_{\bp_{n+1}})$.
Notice that the identity $\bp_{n+1}\approx x^3$ is satisfied by $S_{\bp_n}\circ S_{\bp_n}$,
but does not hold in $S_{\bp_{n+1}}$.
Hence $\mathsf{V}(S_{\bp_n}\circ S_{\bp_n})$ is a proper subvariety of $\mathsf{V}(S_{\bp_{n+1}})$.

(iii) We know that $S_{\bc_1}$ is a subalgebra of $S_{\bc_1}\circ S_{\bc_1}$.
So $\mathsf{V}(S_{\bc_1})$ is a subvariety of $\mathsf{V}(S_{\bc_1}\circ S_{\bc_1})$.
It is easy to check that the identity $x^2+xy\approx x^2+y^2$ is satisfied by $S_{\bc_1}$,
but does not hold in $S_{\bc_1}\circ S_{\bc_1}$.
Thus $\mathsf{V}(S_{\bc_1})$ is a proper subvariety of $\mathsf{V}(S_{\bc_1}\circ S_{\bc_1})$.

To show that $\mathsf{V}(S_{\bc_1}\circ S_{\bc_1})=\mathsf{V}\left(\bigcup_{1\leq i\leq m}^{\omega}S_{\bc_1}\right)$,
it suffices to prove that $\bigcup_{1\leq i\leq m}^{\omega}S_{\bc_1}$ lies in $\mathsf{V}(S_{\bc_1}\circ S_{\bc_1})$.
Let $\{0, \omega, a\}$ and $\{0, \omega, b\}$ be copies of $S_{\bc_1}$.
Then $\{0, \omega, a, b\}$ is isomorphic to $S_{\bc_1}\circ S_{\bc_1}$.

Consider the cyclic shift construction. Formally, let $A$ denote the subalgebra of the direct product $S_{\bc_1}^m$ of $m$ copies of $S_{\bc_1}$ generated by the elements
\[
\sigma^j(a,b, \ldots, b),~ 1\leq j\leq m,
\]
where $\sigma$ is the right cyclic shift mapping from $S_{\bc_1}^m$ to itself.
If we denote by $J$ the set of all elements of $A$ with a $0$-coordinate,
then $J$ is both a multiplicative ideal and an order-theoretic filter of $A$.
Now it is easy to see that $A/J$ is isomorphic to $\bigcup_{1\leq i\leq m}^{\omega} S_{\bc_1}$.
Hence $\mathsf{V}(S_{\bc_1}\circ S_{\bc_1}) =\mathsf{V}(\bigcup_{1\leq i\leq m}^{\omega}S_{\bc_{1}})$,
and so \eqref{eq25060202} holds.

(iv)
Write $S_{\bc_n}=\{0, \omega, a_1, \ldots, a_n\}$.
Here we embed the direct union into a power of $S_{\bc_n}$ via cyclic shifts.

Let $A$ be the subalgebra of the direct product $S_{\bc_n}^m$ of $m$ copies of $S_{\bc_n}$ generated by the elements
\[
\sigma^j(a_i, a_{i+1}, \ldots, a_{i+1}), ~1\leq i\leq n,~ 1\leq j\leq m,
\]
where $\sigma$ is the right cyclic shift mapping from $S_{\bc_n}^m$ onto itself.
If $J$ denotes the elements of $A$ with a $0$-coordinate, then
the quotient algebra $A/J$ is isomorphic to $\bigcup_{1\leq i\leq m}^{\omega} S_{\bc_n}$.
So $\mathsf{V}(S_{\bc_{n}})=\mathsf{V}(\bigcup_{1\leq i\leq m}^{\omega}S_{\bc_{n}})$.

(v)
Write $S_{\bc_k}=\{0,\omega,a_1,\ldots,a_k\}$ and $S_{\bp_n}=\{0,\omega,b_1,\ldots,b_n\}$.
Let $T$ be the $\{0,\omega\}$-direct union of $S_{\bc_k}$ and $S_{\bp_n}$, and let $A$ be the subalgebra of the direct product $T^m$ of $m$ copies of $T$ generated by the elements
\[
\sigma^j(a_i, b_i,\ldots,b_i), ~1\leq i\leq n,~ 1\leq j\leq m,
\]
where $\sigma$ is the right cyclic shift mapping from $T^m$ onto itself and the subscript $i$ of $a_i$ is taken modulo $k$.
If $J$ denotes the elements of $A$ with a $0$-coordinate, then
the quotient algebra $A/J$ is isomorphic to $\bigcup_{1\leq i\leq m}^{\omega} S_{\bp_n}$.
So $\bigcup_{1\leq i\leq m}^{\omega} S_{\bp_n}\in \mathsf{V}(S_{\bc_{k}}\circ S_{\bp_n})$.
\end{proof}

\begin{proposition}\label{scnsln11}
Let $I$ and $J$ be sets \up(possibly empty\up) of positive integers with $1\notin J$, and let $n\geq 2$ be a positive integer.
\begin{enumerate}[$(\rm i)$]
\item
If $I$ does not contain $1$, then
\[
\textstyle\mathsf{V}\left(\bigcup_{1\leq k\leq n}^{\omega}\left(\bigcup_{i\in I}^{\omega}
S_{\mathbf{c}_i}\circ \bigcup_{j\in J}^{\omega} S_{\bp_i}
 \right)\right)
=\mathsf{V}\left(\bigcup_{i\in I}^{\omega} S_{\mathbf{c}_i}\circ \bigcup_{j\in J}^{\omega} S_{\bp_i}\circ \bigcup_{j\in J}^{\omega} S_{\bp_i}\right).
\]
\item If $I$ contains $1$, then
\[
\textstyle\mathsf{V}\left(\bigcup_{1\leq k\leq n}^{\omega}\left(\bigcup_{i\in I}^{\omega} S_{\mathbf{c}_i}\circ \bigcup_{j\in J}^{\omega} S_{\bp_i}\right)\right)=\mathsf{V}\left(S_{\bc_1}\circ\bigcup_{i\in I}^{\omega} S_{\mathbf{c}_i}\circ \bigcup_{j\in J}^{\omega} S_{\bp_i}\right).
\]
\end{enumerate}
\end{proposition}

\begin{proof}
One can obtain this result by modifying $T$ appropriately in the proof of Lemma~\ref{scnsln1}.
\end{proof}

For each integer $n \geq 2$, let $\mathcal{V}_n$ denote the subvariety of $\mathbf{NF}_3$
defined by the identity
\begin{equation}\label{id26060101}
\bp_n\approx \bc_n.
\end{equation}
Before stating the next lemma, we recall that the \emph{length} of a path (or cycle) in a graph is the number of edges it contains. Thus for each $n\geq 2$ and $m\geq 1$, the path $\bp_n$ has length $n-1$ while the cycle $\bc_m$ has length $m$.
\begin{lemma}\label{lem25061422}
Let ${\mathbb G}$ be a graph. If the graph semiring $S_{\mathbb G}$ is a member of $\mathcal{V}_n$,
then
\begin{enumerate}[$(\rm i)$]
\item every path of ${\mathbb G}$ has length at most $n-2$;

\item every cycle of ${\mathbb G}$ has length dividing $n$.
\end{enumerate}
\end{lemma}
\begin{proof}
(i) Suppose on the contrary that $\mathbb{G}$ contains a path of length $n-1$.
Then $S_{\bp_n}$ is a subalgebra of $S_{\mathbb{G}}$, and so $S_{\bp_n}$ belongs to $\mathcal{V}_n$.
This implies that $S_{\bp_n}$ satisfies the identity~\eqref{id26060101}.
Consider the substitution $\varphi \colon \{x_1,\ldots,x_n\} \to S_{\bp_n}$ defined by
$\varphi(x_i) = a_i$ for each $1 \leq i \leq n$.
Under this substitution, we have that $\varphi(\bp_n) = \omega$ but $\varphi(\bc_n) = 0$, which leads to a contradiction.
Therefore, every path in $\mathbb{G}$ has length at most $n-2$.

(ii) Assume that $\mathbb{G}$ contains a cycle of length $m$, where $m$ does not divide $n$.
Then $S_{\bc_m}$ is a subalgebra of $S_{\mathbb{G}}$, and so $S_{\bc_m}$ satisfies the identity~\eqref{id26060101}.
By Lemma~\ref{Scnm}, the algebra $S_{\bc_m}$ satisfies the identity $\bc_n \approx x^3$.
Consequently, $S_{\bc_m}$ satisfies the identity $\bp_n \approx x^3$.
Define the substitution $\varphi \colon \{x_1,\ldots,x_n\} \to S_{\bc_m}$ by setting $\varphi(x_i) = a_k$, where $1 \leq i \leq n$, $1 \leq k \leq m$, and $i \equiv k \pmod{m}$.
This substitution yields $\varphi(\bp_n) = \omega$ while $\varphi(x^3) = 0$, again resulting in a contradiction.
We conclude that every cycle in $\mathbb{G}$ must have length dividing $n$.
\end{proof}

\begin{lemma}\label{lemma250614001}
The variety $\mathcal{V}_n$ is generated by the finite flat semiring
\[
\textstyle S_{\bc_1}\circ \bigcup_{k\mid n}^{\omega}S_{\bc_k}
\]
for all $n \geq 2$.
\end{lemma}
\begin{proof}
We first show that $S_{\bc_1}\circ \bigcup_{k\mid n}^{\omega}S_{\bc_k}$ satisfies the identity \eqref{id26060101}.
Indeed, let
\[\textstyle
\varphi:\{x_1,\ldots, x_n\}\to S_{\bc_1}\circ \bigcup_{k\mid n}^{\omega}S_{\bc_k}
\]
be an arbitrary substitution.
If $\varphi(\bp_n)=0$ then $\varphi(\bc_n)=\varphi(\bp_n)+\varphi(x_{n}x_1)=0$.
If $\varphi(\bp_n)\neq 0$ then
\[
\varphi(x_1)\varphi(x_2)+\varphi(x_2)\varphi(x_3)\cdots \varphi(x_{n-1})\varphi(x_n)=\varphi(\bp_n)=\omega.
\]
This implies that $\{\varphi(x_1),\ldots,\varphi(x_n)\}$ is contained in some connected component of
$S_{\bc_1}\circ \bigcup_{k\mid n}^{\omega}S_{\bc_k}$, say $S_{\bc_m}$, where $m$ is a divisor of $n$.
So $\varphi$ can be viewed as a substitution from $\{x_1,\ldots,x_n\}$ to $S_{\bc_m}$.
Let $a_{j_i}$ denote $\varphi(x_i)$, where $1\leq i\leq n$, $1\leq j_i\leq m$.
By \eqref{eq25053003} in the proof of Lemma \ref{Scnm}, we have
$$j_n\equiv j_1+n-1 \equiv j_1-1 \pmod{m}.$$
This implies that $\varphi(x_n)\varphi(x_1)=\omega$, and so $\varphi(\bc_n)=\omega$.
We have shown that the graph semiring $S_{\bc_1}\circ \bigcup_{k\mid n}^{\omega}S_{\bc_k}$ satisfies~\eqref{id26060101}.
Thus $S_{\bc_1}\circ \bigcup_{k\mid n}^{\omega}S_{\bc_k}$ lies in $\mathcal{V}_n$.

Conversely, suppose that $S$ is a finite subdirectly irreducible member of $\mathcal{V}_n$.
Then, by Theorem~\ref{siSg}, $S$ is isomorphic to a graph semiring $S_{\mathbb G}$ for some finite graph ${\mathbb G}$.
It follows from Lemma~\ref{lem25061422} that every path of ${\mathbb G}$ has length at most $n-2$,
and that every cycle of ${\mathbb G}$ has length that divides $n$.
Let $\{{\mathbb G}_i \mid 1\leq i \leq m\}$ denote the set of all connected components of ${\mathbb G}$.
Then every $S_{\mathbb G_i}$ is a subalgebra of $\bigcup_{k\mid n}^{\omega} S_{\bc_k}$.
By Proposition \ref{pro25061320}, $S_{\mathbb G}$ is the $\{0,\omega\}$-direct union of
$\{S_{\mathbb G_i} \mid 1\leq i\leq m\}$.
This implies that $S_{\mathbb G}$ is a subalgebra of the $\{0,\omega\}$-direct union
of $m$ copies of $\bigcup_{k\mid n}^{\omega} S_{\bc_k}$.
By Proposition~\ref{scnsln11}(ii) one can deduce that $S_{\mathbb G}$ lies in $\mathsf{V}(S_{\bc_1}\circ \bigcup_{k\mid n}^{\omega}S_{\bc_k})$,
and so does $S$.
Thus $\mathcal{V}_n$ is generated by $S_{\bc_1}\circ \bigcup_{k\mid n}^{\omega}S_{\bc_k}$.
\end{proof}

\begin{proposition}\label{fgsubvariety}
The variety $\mathcal{V}_n$ is a Cross variety for all $n \geq 2$.
\end{proposition}
\begin{proof}
By Proposition~\ref{fkjidi}(ii) and Lemma \ref{lemma250614001}, one can deduce that $\mathcal{V}_n$ is finitely based
and finitely generated.
In the reminder we shall show that $\mathcal{V}_n$ contains only finitely many critical algebras.
Indeed, let $S$ be a critical algebra in $\mathcal{V}_n$.
Then $S$ must be subdirectly irreducible.
By Theorem \ref{siSg}, $S$ is isomorphic to some finite graph semiring $S_{\mathbb G}$.
By Lemma~\ref{lem25061422}, every path of ${\mathbb G}$ has length at most $n-2$,
and that every cycle of ${\mathbb G}$ has length that divides $n$.
Since $S_{\mathbb{G}}$ is critical, Proposition~\ref{scnsln11} implies that in $\mathbb{G}$,
loops and paths of equal length appear at most twice, and cycles of length $>1$ are distinct.
Thus the number of such graphs is finite and so $\mathcal{V}_n$ contains only finitely many critical algebras.
Therefore, $\mathcal{V}_n$ is a Cross variety.
\end{proof}


\begin{proposition}\label{pro25061501}
Every finitely generated subvariety of $\mathbf{NF}_3$ is contained in
$\mathcal{V}_n$ for some integer $n\geq 2$.
\end{proposition}
\begin{proof}
We first show that every finite graph semiring belongs to the variety $\mathcal{V}_n$ for some integer $n\geq 2$.
Indeed, let $S_{\mathbb G}$ be an arbitrary finite graph semiring.
Suppose that the maximum length of paths and cycles in $\mathbb G$ is less than $m$.
Let $\{{\mathbb G}_i \mid 1\leq i \leq k\}$ denote the set of all connected components of ${\mathbb G}$.
Then $S_{\mathbb G_i}$ is a subalgebra of $\bigcup_{1\leq i\leq m}^{\omega} S_{\bc_i}$ for every $1\leq i \leq k$.
By Proposition \ref{pro25061320}, $S_{\mathbb G}$ is the $\{0,\omega\}$-direct union of
$\{S_{\mathbb G_i} \mid 1\leq i\leq k\}$.
This implies that $S_{\mathbb G}$ is a subalgebra of the $\{0,\omega\}$-direct union
of $k$ copies of $\bigcup_{1\leq i\leq m}^{\omega} S_{\bc_i}$.
Now it follows from Proposition \ref{scnsln11}(ii) that $S_{\mathbb G}$ lies in the variety
$\mathsf{V}(S_{\bc_1}\circ \bigcup_{1\leq i\leq m}^{\omega} S_{\bc_i})$.
Take $n=m!$. By Lemma \ref{lemma250614001},
$S_{\bc_1}\circ \bigcup_{1\leq i\leq m}^{\omega} S_{\bc_i}$ is a member of $\mathcal{V}_{n}$.
So $S_{\mathbb G}$ is in $\mathcal{V}_{n}$.

Let $\mathcal{V}$ be a finitely generated subvariety of $\mathbf{NF}_3$.
It follows from Lemma~\ref{lem24121301} and Theorem~\ref{siSg} that
$\mathcal{V}$ is generated by finitely many finite graph semirings $S_1, \ldots, S_m$.
By the preceding result, there exists $n$ such that
the finite graph semiring $\bigcup^\omega_{1\leq i \leq m} S_{i}$
lies in $\mathcal{V}_n$.
This implies that $\mathcal{V}$ is a subvariety of $\mathcal{V}_n$.
\end{proof}

\begin{corollary}\label{fgsubkehua}
Let $\mathcal{V}$ be subvariety of $\mathbf{NF}_3$.
If $\mathcal{V}$ does not contain $S_{\bp_m}$ for some $m\geq 2$, then $\mathcal{V}$ is finitely generated.
\end{corollary}
\begin{proof}
Let $S_{\mathbb{G}}$ be an arbitrary graph semiring in $\mathcal{V}$.
Since $S_{\bp_m}$ is not a subalgebra of $S_{\mathbb{G}}$,
the maximum length of paths and cycles in $\mathbb{G}$ must be less than $m$.
From the proof of Proposition~\ref{pro25061501}, it follows that $S_{\mathbb{G}}$ is a member of $\mathcal{V}_{m!}$,
and so $\mathcal{V}$ is a subvariety of $\mathcal{V}_{m!}$.
By Proposition~\ref{fgsubvariety} we therefore conclude that $\mathcal{V}$ is finitely generated.
\end{proof}

\begin{thm}\label{coro25060115}
Every finitely generated subvariety of $\mathbf{NF}_3$ is a Cross variety.
\end{thm}
\begin{proof}
This follows from Propositions \ref{fgsubvariety} and \ref{pro25061501} immediately.
\end{proof}

\begin{corollary}\label{coro25060310}
Every finite $3$-nilpotent flat semiring is finitely based.
\end{corollary}
\begin{proof}
This is a consequence of Theorem~\ref{coro25060115}.
\end{proof}


\begin{corollary}
The variety $\mathbf{NF}_3$ is not finitely generated.
\end{corollary}
\begin{proof}
Suppose by way of contradiction that $\mathbf{NF}_3$ is finitely generated.
By Theorem~\ref{coro25060115} we have that $\mathbf{NF}_3$ is a Cross variety.
This implies that $\mathbf{NF}_3$ has finitely many subvarieties, which contradicts Theorem~\ref{cor:continuum}.
Hence $\mathbf{NF}_3$ is not finitely generated as required.
\end{proof}

The following example shows that a finitely based subvariety of $\mathbf{NF}_3$ is not necessarily finitely generated.
\begin{example}
$\mathsf{V}(\bigcup_{n\geq 2}^{\omega}S_{c_n})$ is determined within the variety $\mathbf{NF}_3$
by the identity $x^2\approx x^3$, but is not finitely generated.
In fact, let $\mathcal W$ denote the subvariety of $\mathbf{NF}_3$ determined by the identity $x^2\approx x^3$.
It is easy to see that $\bigcup_{n\geq 2}^{\omega}S_{c_n}$ satisfies $x^2\approx x^3$,
and so $\mathsf{V}(\bigcup_{n\geq 2}^{\omega}S_{c_n})$ is a subvariety of $\mathcal W$.
Conversely, let $S$ be a finite nontrivial subdirectly irreducible member of $\mathcal W$.
It follows from Theorem $\ref{siSg}$ and Lemma $\ref{Scnm}$ that $S$ is isomorphic to a finite graph semiring $S_{\mathbb G}$,
where $\mathbb G$ is a finite loop-free graph.
So $S_{\mathbb G}$ is a subalgebra of $\{0,\omega\}$-direct union of $m$ copies of $\bigcup_{n\geq 2}^\omega S_{\bc_n}$,
where $m$ is the number of connected components of $\mathbb G$.
By Proposition~$\ref{scnsln11}(\rm i)$, we have that $S_{\mathbb G}$ is a member of
$\mathsf{V}(\bigcup_{n\geq 2}^{\omega}S_{c_n})$,
and so is $S$.
Thus $\mathsf{V}(\bigcup_{n\geq 2}^{\omega}S_{c_n})=\mathcal W$ as required.
\end{example}

Proposition~\ref{pro25061501} essentially provides a general method for finding a finite equational basis of a finitely generated subvariety $\mathcal{V}$ of  $\mathbf{NF}_3$.
Firstly, determine $n$ such that $\mathcal{V}$ is a subvariety of $\mathcal{V}_n$.
Next, find a maximal chain from $\mathcal{V}$ to $\mathcal{V}_n$, which is denoted by
\[
\mathcal{V} = \mathcal{W}_0 < \mathcal{W}_1 < \cdots < \mathcal{W}_m = \mathcal{V}_n.
\]
For any $1 \leq k \leq m$, let us take an identity
$\bu_k \approx \bv_k$ that is satisfied by $\mathcal{W}_{k-1}$, but is not true in $\mathcal{W}_{k}$.
So $\mathcal{V}$ is the subvariety of $\mathbf{NF}_3$ determined by the following finitely many identities:
\[
\bp_n\approx \bc_n, \bu_1 \approx \bv_1, \ldots, \bu_m \approx \bv_m.
\]

We end this section by showing that finding the finite equational basis of $\mathsf{V}(S_{\bc_n})$
can also be addressed through \emph{forbidden identities}.

\begin{lemma}\label{sckk|njinqu}
Let $\mathcal{V}$ be a subvariety of $\mathbf{NF}_3$.
Then ${\mathcal V}$ satisfies the identity $\bc_n\approx x^3$ if and only if
${\mathcal V}$ does not contain $S_{\bc_k}$ for all divisors $k$ of $n$.
\end{lemma}

\begin{proof}
Suppose that ${\mathcal V}$ contains a graph semiring $S_{\bc_k}$ for some divisor $k$ of $n$.
Then by Lemma~\ref{Scnm}, we have that ${\mathcal V}$ does not satisfy $\bc_n\approx x^3$.

Conversely, assume that ${\mathcal V}$ does not satisfy $\bc_n\approx x^3$.
By Theorem~\ref{siSg}, there is a finite graph semiring $S_{\mathbb G}$ in $\mathcal{V}$
such that $\varphi(\bc_n)\neq \varphi(x^3)$ for some substitution  $\varphi:\{x, x_1, \ldots, x_n\}\to S_{\mathbb G}$.
Since $\varphi(x^3)=0$, we have
\[
\varphi(x_1)\varphi(x_2)+\varphi(x_2)\varphi(x_3)+\cdots+\varphi(x_{n-1})\varphi(x_n)+\varphi(x_{n})\varphi(x_1)=\varphi(\bc_n)=\omega.
\]
and so $\{\varphi(x_1), \ldots, \varphi(x_n)\}$ is contained in some connected component $\mathbb G_i$ of $\mathbb G$.
It is obvious that $\mathbb G_i$ is a circle graph, and so $S_{\mathbb G_i}$ is isomorphic to $S_{\bc_k}$ for some $k$.
Since $S_{\bc_k}$ does not satisfy the identity $\bc_n\approx x^3$, it follows from Lemma~\ref{Scnm} again that $k$ is a divisor of $n$.
So ${\mathcal V}$ contains $S_{\bc_k}$ for some divisor $k$ of $n$.
\end{proof}

\begin{proposition}
Let $n\geq 2$ be an integer. Then $\mathsf{V}(S_{\bc_n})$ is the subvariety of $\mathbf{NF}_3$
determined by the identities $\bp_n\approx \bc_n$ and $\bc_k\approx x^3$, where $k$ runs through all positive divisors of $n$ with $k\neq n$.
\end{proposition}

\begin{proof}
Let $\mathcal{V}$ be the subvariety of $\mathbf{NF}_3$
determined by $\bp_n\approx \bc_n$ and $\bc_k\approx x^3$, where $k$ runs through the proper divisors of $n$.
It follows from Lemmas~\ref{Scnm} and \ref{lemma250614001}
that $S_{\bc_n}$ is a member of $\mathcal{V}$, and so $\mathsf{V}(S_{\bc_n})$ is a subvariety of $\mathcal{V}$.
Conversely, let $S_{\mathbb G}$ be a finite graph semiring in $\mathcal{V}$.
By Lemmas~\ref{lem25061422} and \ref{sckk|njinqu}, we have that every path in $\mathbb G$ has length at most $n-2$ and every cycle in $\mathbb G$ is of length $n$.
This implies that $S_{\mathbb G}$ is a subalgebra of the $\{0,\omega\}$-direct union of some copies of
$S_{\bc_{n}}$.
By Proposition~\ref{scnsln11}(i) we have that $S_{\mathbb G}$ is a member of $\mathsf{V}(S_{\bc_n})$.
Thus $\mathsf{V}(S_{\bc_n})=\mathcal{V}$.
\end{proof}

\section{The limit subvariety of $\mathbf{NF}_3$}
In this section we show that $\mathbf{NF}_3$ has a unique limit subvariety.
Let us denote by $S_{\bp_{\infty}}$ the graph semiring induced by a path graph of infinite length,
that is, a graph with the vertex set $\{a_k \mid k \geq 1\}$ and the edge set $\{(a_k, a_{k+1}) \mid k \geq 1\}$.
It is straightforward to verify that each $S_{\mathbf{p}_n}$ forms a subalgebra of $S_{\mathbf{p}_\infty}$.
Conversely, every finitely generated subalgebra of $S_{\mathbf{p}_\infty}$ is contained in some $S_{\mathbf{p}_n}$.
This implies that $S_{\mathbf{p}_\infty}$ and $\{S_{\bp_n}\mid n\geq 1\}$ satisfy the same identities.
So we have

\begin{pro}\label{pro25060725}
$\mathsf{V}(S_{\bp_{\infty}})=\mathsf{V}(\{S_{\bp_n}\mid n\geq 1\})$.
\end{pro}

A graph semiring $S_{\mathbb G}$ is \emph{acyclic} if the graph $\mathbb G$ is acyclic, that is, $\mathbb G$ contains no cycles.
One can easily see that $S_{\bp_{\infty}}$ is an example of an acyclic graph semiring.
We shall use $\mathcal{V}_{ac}$ to denote the subvariety of $\mathbf{NF}_3$ generated by all acyclic graph semirings.

%
%
%

\begin{pro}\label{pro25060701}
The variety $\mathcal{V}_{ac}$ is generated by $S_{\bp_{\infty}}$.
\end{pro}
\begin{proof}
Let $S_{\mathbb G}$ be an acyclic graph semiring,
and let $\{\mathbb G_i\}_{i\in I}$ denote the set of all connected components of $\mathbb G$.
It is easy to see that $S_{\mathbb G}$ is a subalgebra of $\bigcup_{i\in I}^{\omega}S_{\bp_{\infty}}$.
Let $T$ be the algebra of the direct product $\prod_{i\in I} S_{\bp_{\infty}}$ generated by the subset $\{\alpha_{ik}\mid i\in I, k\geq 1\}$, where for each $j\in I$,
\[
\alpha_{ik}(j)=\begin{cases}
  a_k,& j=i,\\
  a_{k+1},& j\neq i.
\end{cases}
\]

If we denote by $J$ the set of all elements of $T$ with a $0$-coordinate,
then $J$ is both a multiplicative ideal and an order-theoretic filter of $T$.
Now it is a routine matter to verify that the quotient algebra $T/J$ is isomorphic to $\bigcup_{i\in I}^{\omega}S_{\bp_{\infty}}$,
and so $S_{\mathbb G}$ is a member of $\mathsf{V}(S_{\bp_{\infty}})$.
Thus $\mathcal{V}_{ac}$ is generated by $S_{\bp_{\infty}}$.
\end{proof}

\begin{lemma}\label{vacbasis}
The variety $\mathcal{V}_{ac}$ is determined within $\mathbf{NF}_3$ by the identities
\begin{equation} \label{bcnx3n1}
\bc_n\approx x^3,~n \geq 1.
\end{equation}
\end{lemma}
\begin{proof}
Let $\mathcal{W}$ be the subvariety of $\mathbf{NF}_3$ determined  by the identities \eqref{bcnx3n1}.
It is easy to see that $S_{\bp_{\infty}}$ satisfies the identities \eqref{bcnx3n1}.
By Proposition \ref{pro25060701} we have that $\mathcal{V}_{ac}$ is a subvariety of $\mathcal{W}$.
Conversely, let $S_{\mathbb G}$ be a finite graph semiring in $\mathcal{W}$.
Suppose by way of contradiction that $\mathbb G$ is not acyclic.
Then $\mathbb G$ contains some cycle of length $m$, and so $S_{\bc_m}$ lies in $\mathcal{W}$.
This implies that $S_{\bc_m}$ satisfies $\bc_m\approx x^3$, which is a contradiction.
Thus $\mathbb G$ is acyclic and so $S_{\mathbb G}$ is a acyclic graph semiring.
Hence $\mathcal{W}$ is a subvariety of $\mathcal{V}_{ac}$
and so $\mathcal{V}_{ac}=\mathcal{W}$.
\end{proof}

\begin{corollary}\label{everyacyclic}
  Every graph semiring in $\mathcal{V}_{ac}$ is acyclic.
\end{corollary}

\begin{pro}\label{vacnfb}
The variety $\mathcal{V}_{ac}$ is nonfinitely based.
\end{pro}
\begin{proof}
Assume, for contradiction, that $\mathcal{V}_{ac}$ is finitely based.
It follows from Lemma~\ref{vacbasis} and the Compactness Theorem for Equational Logic~\cite[Exercise 10 in \S14]{bs}
that $\mathcal{V}_{ac}$ is the subvariety of $\mathbf{NF}_3$ determined by the identities
\begin{equation}\label{ckx31km}
  \bc_k\approx x^3,~ 1\leq k\leq m
\end{equation}
for some $m$.
By Lemma~\ref{Scnm}, the cycle graph semiring $S_{\bc_{m+1}}$ satisfies the identities \eqref{ckx31km}.
This implies that $S_{\bc_{m+1}}$ is a member of $\mathcal{V}_{ac}$, which contradicts Corollary \ref{everyacyclic}.
So $\mathcal{V}_{ac}$ is nonfinitely based.
\end{proof}

\begin{corollary}\label{coro25060720}
The variety $\mathcal{V}_{ac}$ is not finitely generated.
\end{corollary}
\begin{proof}
This follows from Corollary \ref{coro25060310} and Proposition \ref{vacnfb} immediately.
\end{proof}

\begin{pro}\label{lvvac}
The subvariety lattice of the variety $\mathcal{V}_{ac}$ is
\[
\mathbf{T}<\mathsf{V}(S_{\bp_1})<\mathsf{V}(S_{\bp_1}\circ S_{\bp_1})<\cdots <\mathsf{V}(S_{\bp_n})<\mathsf{V}(S_{\bp_n}\circ S_{\bp_n})<\cdots<\mathcal{V}_{ac},
\]
which is isomorphic to $\boldsymbol{\omega}\oplus \mathbf{1}$, the linear sum of the chain $\boldsymbol\omega$ of natural numbers under the usual order and the trivial one-element lattice $\mathbf{1}$ corresponds to $\mathcal{V}_{ac}$.
\end{pro}

\begin{proof}
Let $\mathcal{W}$ be a subvariety of $\mathcal{V}_{ac}$.
It follows from Corollary~\ref{everyacyclic} that $\mathcal{W}$ can be
generated by a family $(S_{\mathbb G_i})_{i\in I}$ of finite acyclic graph semirings.
For any $i \in I$, let $n_i$ denote the maximum of lengths of all paths in $\mathbb G_i$.
Consider the following two cases:

\textbf{Case 1.} The set $\{n_i \mid i\in I\}$ has a maximum element $n$.

\textbf{Subcase 1.1.} There are at least two paths of length $n$ in some graph $\mathbb G_i$.
It is obvious that $S_{\bp_n}\circ S_{\bp_n}$ is a member of $\mathcal{W}$,
and so $\mathsf{V}(S_{\bp_n}\circ S_{\bp_n})$ is a subvariety of $\mathcal{W}$.
On the other hand, for any $i\in I$, we may assume that $\mathbb G_i$ has exactly $n_i$ connected components.
Then $S_{\mathbb G_i}$ is a subalgebra of  $\bigcup_{1\leq k\leq n_i}^{\omega} S_{\bp_n}$.
It follows from Lemma~\ref{scnsln1}\eqref{eq250602011} that $S_{\mathbb G_i}$ lies in $\mathsf{V}(S_{\bp_n}\circ S_{\bp_n})$.
So $\mathcal{W} = \mathsf{V}(S_{\bp_n}\circ S_{\bp_n})$.

\textbf{Subcase 1.2.} There is at most one path of length $n$ in any graph $\mathbb G_i$.
It is easy to see that $S_{\bp_n}$ lies in $\mathcal{W}$.
Conversely, for any $i\in I$, suppose that $\mathbb G_i$ has exactly $n_i$ connected components.
Then $S_{\mathbb G_i}$ is a subalgebra of  $S_{\bp_{n}}\circ \bigcup_{1\leq k\leq n_i}^{\omega} S_{\bp_{n-1}}$.
It follows from Lemma~\ref{scnsln1}\eqref{eq250602012} that $S_{\mathbb G_i}$ is a member of $\mathsf{V}(S_{\bp_n})$.
Thus $\mathcal{W} = \mathsf{V}(S_{\bp_n})$.

\textbf{Case 2.} The set $\{n_i \mid i\in I\}$ has no upper bound.
It is easy to see that $S_{\bp_n}$ lies in $\mathcal{W}$ for all $n\geq 2$.
By Proposition \ref{pro25060725}, we have that $\mathcal{W}=\mathcal{V}_{ac}$.

So we have shown that the subvariety lattice of the variety $\mathcal{V}_{ac}$ consists of $\mathcal{V}_{ac}$,
$\mathsf{V}(S_{\bp_n})$ and $\mathsf{V}(S_{\bp_n}\circ S_{\bp_n})$ for all $n\geq 1$.
By Lemma~\ref{scnsln1}(i) one can deduce that
\[
\mathbf{T}<\mathsf{V}(S_{\bp_1})<\mathsf{V}(S_{\bp_1}\circ S_{\bp_1})<\cdots <\mathsf{V}(S_{\bp_n})<\mathsf{V}(S_{\bp_n}\circ S_{\bp_n})<\cdots<\mathcal{V}_{ac}.
\]
Therefore, the subvariety lattice of $\mathcal{V}_{ac}$ is isomorphic to
the linear sum of the lattice of all natural numbers under the usual order and the trivial one-element lattice.
\end{proof}

\begin{corollary}\label{coro250607001}
Every proper subvariety of $\mathcal{V}_{ac}$ is finitely based.
\end{corollary}
\begin{proof}
This is a consequence of Proposition \ref{lvvac} and Corollary \ref{coro25060310}.
\end{proof}

By Proposition \ref{vacnfb} and Corollary \ref{coro250607001} we immediately obtain
\begin{pro}\label{pro25060715}
$\mathcal{V}_{ac}$ is a limit variety.
\end{pro}
\begin{remark}
The limit varieties in~\cite{rjzl} are either group based, contain only degenerate flat semirings (by \cite[Theorem 4.1]{rjzl}) or are easily seen to satisfy \(\bp_{3}\approx x^3\) (in the case of examples from \cite[Section 4.2]{rjzl}). It follows that the limit variety \(\mathcal{V}_{ac}\) is distinct from all of these, and so is a new example.

\end{remark}
\begin{pro}\label{fgsvvac}
Let $\mathcal{V}$ be a subvariety of $\mathbf{NF}_3$.
Then $\mathcal{V}$ is finitely generated if and only if it does not contain $\mathcal{V}_{ac}$.
\end{pro}
\begin{proof}
Let $\mathcal{V}$ be a finitely generated subvariety of $\mathbf{NF}_3$.
By Proposition~\ref{fgsubvariety}, $\mathcal{V}$ is contained in some $\mathcal{V}_n$,
and so $\mathcal{V}$ satisfies the identity~\eqref{id26060101}.
Since the path graph semiring $S_{\bp_{n}}$ fails to satisfy the identity~\eqref{id26060101},
it follows that $S_{\bp_n}$ does not belong to $\mathcal{V}$.
Therefore, $\mathcal{V}$ does not contain $\mathcal{V}_{ac}$.

Conversely, suppose that $\mathcal{V}_{ac}$ is not contained in $\mathcal{V}$.
By Proposition~\ref{pro25060725}, there exists an integer $m\geq 2$ such that $S_{\bp_m}$ does not lie in $\mathcal{V}$.
It then follows from Corollary~\ref{fgsubkehua} that $\mathcal{V}$ is finitely generated.
\end{proof}

\begin{remark}
Corollary $\ref{coro25060720}$ and Proposition~$\ref{fgsvvac}$ together show that $\mathcal{V}_{ac}$
is the smallest not-finitely generated subvariety of $\mathbf{NF}_3$.
\end{remark}
Now we are ready to prove the main result of this section.
\begin{thm}
$\mathcal{V}_{ac}$ is the unique limit subvariety of $\mathbf{NF}_3$.
\end{thm}

\begin{proof}
From Proposition~\ref{pro25060715} we know that $\mathcal{V}_{ac}$ is a limit variety.
Let $\mathcal{V}$ be an arbitrary nonfinitely based subvariety of $\mathbf{NF}_3$.
By Corollary~\ref{coro25060115} we deduce that $\mathcal{V}$ is not finitely generated.
It follows from Proposition~\ref{fgsvvac} that $\mathcal{V}_{ac}$ is a subvariety of $\mathcal{V}$.
Thus $\mathcal{V}_{ac}$ is the unique limit subvariety of $\mathbf{NF}_3$ as required.
\end{proof}

\begin{corollary}
The variety $\mathcal{V}_{ac}$ is not a Cross variety, but every proper subvariety of $\mathcal{V}_{ac}$ is a Cross variety.
\end{corollary}

We end this section by providing finite equational bases for $\mathsf{V}(S_{\bp_n}\circ S_{\bp_n})$
and $\mathsf{V}(S_{\bp_n})$, although both of them have been known to be finitely based.

\begin{pro}\label{bpnbpnfb}
$\mathsf{V}(S_{\bp_n}\circ S_{\bp_n})$ is the subvariety of $\mathbf{NF}_3$
determined by the identities
\begin{equation}\label{id25060801}
\bp_{n+1}\approx x^3.
\end{equation}
\end{pro}
\begin{proof}
It is easy to check that $S_{\bp_n}\circ S_{\bp_n}$ satisfies \eqref{id25060801}.
In the remainder it is enough to show that every finite graph semiring that satisfies \eqref{id25060801}
lies in $\mathsf{V}(S_{\bp_n}\circ S_{\bp_n})$.
Let $S_{\mathbb G}$ be such an algebra.
Then the graph $\mathbb G$ is acyclic and has no path of length greater than $n-1$.
By Lemma ~\ref{scnsln1}(i),
we have that $S_{\mathbb G}$ is a member of $\mathsf{V}(S_{\bp_n}\circ S_{\bp_n})$.
Thus $\mathsf{V}(S_{\bp_n}\circ S_{\bp_n})$ is the subvariety of $\mathbf{NF}_3$ determined by \eqref{id25060801} as required.
\end{proof}

\begin{pro}\label{bpnfb}
$\mathsf{V}(S_{\bp_n})$ is the subvariety of $\mathbf{NF}_3$ determined by the identities
\eqref{id25060820} and \eqref{id25060801}.
\end{pro}
\begin{proof}
From Proposition \ref{bpnbpnfb} we know that $\mathsf{V}(S_{\bp_n}\circ S_{\bp_n})$ is the subvariety of $\mathbf{NF}_3$
defined by \eqref{id25060801}.
By proposition~\ref{lvvac}, $\mathsf{V}(S_{\bp_n})$ is a maximal subvariety of $\mathsf{V}(S_{\bp_n}\circ S_{\bp_n})$.
Notice that the identity \eqref{id25060820} is true in $S_{\bp_n}$,
but does not hold in $S_{\bp_n}\circ S_{\bp_n}$.
It follows that $\mathsf{V}(S_{\bp_n})$ is determined within $\mathsf{V}(S_{\bp_n}\circ S_{\bp_n})$ by \eqref{id25060820}.
So the required result is true.
\end{proof}

\section{Conclusion}
We have shown that every finite $3$-nilpotent flat semiring is finitely based.
This partially solves an open problem proposed by Jackson et al.~\cite{jrz}.
However, the finite basis problem for general finite nilpotent flat semirings remains unresolved, even in the case of $S(W)$.
Unlike the $3$-nilpotent case, certain finite $4$-nilpotent flat semirings are nonfinitely based,
with $S_c(abc)$ being a specific example, which can be deduced by~\cite[Theorem 4.9]{jrz}.
This demonstrates the substantial challenges in obtaining a complete solution to the finite basis problem for finite nilpotent flat semirings.

Our classification of the finite basis property for $3$-nilpotent flat ai-semirings is unlikely to be extended to the $4$-nilpotent case. At the 2024 Theoretical and Computational Algebra conference (in celebration of the 75th birthday of John Meakin),
Marcel Jackson presented an overview of the NP-hardness of the finite basis property for $4$-nilpotent flat semirings. It would then follow that no tractable classification in the $4$-nilpotent case is possible (assuming $\texttt{P}\neq \texttt{NP}$).

\quad

\noindent
\textbf{Acknowledgements}
The authors are deeply grateful to Professor Pavel Kolesnikov for his generous and repeated assistance at various stages of this work.
We would also like to extend our sincere thanks to Professor Marcel Jackson for his comments and suggestions, which were particularly helpful in shaping the discussion in the Conclusion section.
Finally, we wish to extend our heartfelt thanks to the anonymous referees,
whose painstaking efforts in reviewing the manuscript and whose insightful observations and suggestions have greatly enhanced the clarity and quality of this paper.

\quad

\noindent
{\bf Disclosure Statement}
The authors do not have any relevant financial or non-financial competing interests.

\bibliographystyle{amsplain}

\end{document}